\documentclass{amsart}

\usepackage[english]{babel}
\usepackage{amssymb,graphicx,bbm,amsthm,graphicx}
\usepackage{mathtools}
\usepackage{epstopdf}
\DeclareGraphicsExtensions{.eps,.pdf}

\newcommand{\cdummy}{\cdot}
\newcommand{\nin}{\not\in}
\newcommand{\nocomma}{}
\newcommand{\noplus}{}
\newcommand{\tmem}[1]{{\em #1\/}}
\newcommand{\tmop}[1]{\ensuremath{\operatorname{#1}}}
\newcommand{\tmstrong}[1]{\textbf{#1}}
\newcommand{\tmtextit}[1]{{\itshape{#1}}}
\newcommand{\tmverbatim}[1]{{\ttfamily{#1}}}
\newenvironment{itemizedot}{\begin{itemize}
    }{\end{itemize}}

\newtheorem{theorem}{Theorem}[section]
\newtheorem{corollary}[theorem]{Corollary}
\newtheorem{definition}[theorem]{Definition}
\newtheorem{lemma}[theorem]{Lemma}
\newtheorem{proposition}[theorem]{Proposition}

\theoremstyle{definition}
\newtheorem{example}[theorem]{Example}
\theoremstyle{remark}
\newtheorem{remark}[theorem]{Remark}

\usepackage{hyperref}

\begin{document}

\title{Cayley-type graphs for group-subgroup pairs}

\author{Cid Reyes-Bustos}

\date{\today}


\begin{abstract}
  In this paper we introduce a Cayley-type graph for group-subgroup
  pairs and present some elementary properties of such graphs, including
  connectedness, their degree and partition structure, and
  vertex-transitivity. We relate these properties to those of the underlying
  group-subgroup pair. From the properties of the group, subgroup and
  generating set some of the eigenvalues can be determined, including the
  largest eigenvalue of the graph. In particular, when this construction
  results in a bipartite regular graph we show a sufficient condition on the
  size of the generating sets that results on Ramanujan graphs for a fixed
  group-subgroup pair. Examples of Ramanujan pair-graphs that do not
  satisfy this condition are also provided, to show that the
  condition is not necessary.
\end{abstract}


{\maketitle}

\section{Introduction }

A Ramanujan graph is a $k$-regular finite graph with nontrivial eigenvalues
$\mu$ satisfying
\begin{equation}\label{eqn:ramprop}
  |\mu| \leqslant 2 \sqrt{k-1} ,
\end{equation}
this condition is equivalent to the ``Graph Theoretical Riemann Hypothesis''
for the Ihara zeta function associated to the graph, see {\cite{Terras2011}} for more
information. The above equivalence was first formulated by Sunada in {\cite{Sunada1988}}.
Infinite families of $k$-regular Ramanujan graphs of increasing size are examples of the
more general families of expander graphs. In fact, the families of Ramanujan graphs are the
expander families that are optimal from the spectral point of view, according to the bounds
by Alon and Boppana. Furthermore, the known constructions of families of Ramanujan graphs are
essentially number theoretical; for instance, the proof of the Ramanujan property (\ref{eqn:ramprop})
for the original families is related to the Ramanujan conjecture for modular forms of weight 2.
The families of expanders have applications to engineering and computer science, among others areas
\cite{Hoory2006}.

The original construction of families of Ramanujan graphs was presented by
Lubotzky, Phillips and Sarnak in
{\cite{Lubotzky1988}}{\tmem{{\tmem{{\tmstrong{}}}}}}, and independently by
Margulis in {\cite{Margulis1988}}. The former construction consists of Cayley
graphs on projective linear groups over a finite field. The use of Cayley
graphs allows to utilize the properties of the underlying group and the
generating set to determine structural and spectral properties of the graph.
For example, the irreducible representations of the underlying group are
directly related to the eigenvalues of the graph.

Recently, there has been work on generalizing the notion of group determinant
as a function of a group-subgroup pair using the representation theory of
$\alpha$-determinants {\cite{Kimoto2008}}. The resulting wreath determinant
for group-subgroup pairs shares some of the properties of the group
determinant, but there is a nontrivial relation with the chosen ordering of
the group and subgroup elements. Some results on factorization for this
determinant can be found on {\cite{Hamamoto2014}}, it turns out to be closely
related to the representation of the symmetric groups associated with the
given group and subgroup and their wreath product.

It was suggested to the author by Masato Wakayama that one could employ the
same strategy to extend existing constructions, or to conceive new
constructions, of Ramanujan graphs and expanders based on groups, in
particular, by generalizing Cayley graphs. In this paper, we apply these ideas
to introduce an extension of Cayley graphs for a group-subgroup pair. The
motivation is twofold, one is to have further tools to solve problems, in
particular the construction of Ramanujan graphs; and to continue the work on
this incipient group-subgroup study philosophy.

The main purpose of this paper is to introduce the concept of group-subgroup
pair graph based on these considerations and show its elementary
properties. In particular, the degree of the vertices of the pair
graph is determined by the coset structure of the subgroup and the
chosen generating set.
The organization of the paper is as follows. In Section 2, the definition,
examples and basic properties are presented. Namely, the determination of the
degrees of the vertices of the graph, and the conditions for the graph to be
regular. In Section 3, conditions for connectedness of the graph are
completely determined, along with the number and properties of connected
components for disconnected graphs. Additionally, sufficient conditions for
the graphs to be bipartite and the notion of trivial eigenvalues are
introduced for the group-subgroup pair graphs. In Section 4 we limit the
discussion to the regular graph case and assert a symmetry relation between
the spectrum for different choices of generating set for fixed group and
subgroup, this result is applied to give a sufficient condition on the size of
the generating set that guarantees that the resulting graphs are Ramanujan
graphs. We provide two examples of Ramanujan graphs obtained by using this
condition. The common theme throughout this paper is to compare and relate the
results on group-subgroup pair graphs with results on Cayley graphs.

\section{Definition and basic properties }

In this section we introduce the group-subgroup pair graph, a
Cayley-type graph construction for group-subgroup pairs; relate the
definition with the group matrix and establish some elementary properties of the graphs.

First, we introduce the conventions and notation used in this paper. All
groups are assumed to be finite unless otherwise stated and $e$ always
represents the identity of a given group G. For a subgroup $H$ of $G$, all
cosets are assumed to be right cosets and we say that $a$ is incongruent to
$b$ modulo $H$ when $H a  \neq H b$. A subset $X \subset G$ is said to be
symmetric if $X^{-1} =X$. For a given group $G$ and symmetric subset $S$ we
denote the corresponding Cayley graph by $\mathcal{G} ( G,S )$. The
characteristic function of a subset $X  \subset G$ is denoted $\delta_{X}$ and
the notation $[ k ]$ for $k \in \mathbbm{N}$ is used for the set $\{ 1,2,
\ldots ,k \}$.

\newpage

\subsection{Definition and examples}

\begin{definition}
  \label{def}Let $G$ be a group, $H$ a subgroup and $S  \subset  G$ a subset
  such that $S \cap H $ is a symmetric subset of $G$. Then the \textbf{Group-Subgroup Pair Graph} $\mathcal{G} ( G,H,S )$ is the graphwith vertex set G and edges
  \[
  \begin{dcases}
    ( h,h s ), ( h s,h  ) & \forall  h  \in H,\hspace{2mm} \forall s
    \in S - H , \\
    ( h, h s )  & \forall h \in H  \hspace{2mm}  \forall s \in S \cap
    H
  \end{dcases}
  \]

\end{definition}

Equivalently, the group-subgroup pair graph $\mathcal{G} ( G,H,S )$ can be
defined as
\[ \mathcal{G} ( G,H,S )  =  \bar{\mathcal{G}} ( G,H,S_{O}   )   \oplus
\mathcal{G} ( H,S_{H} ) , \]
where $S_{O}  = S -H$, $S_{H}  = S  \cap  H$ and $\oplus$ is the generalized
edge sum operator, as defined in {\cite{Knauer2011}}. The graph $\mathcal{G} (
H,S_{H} )$ is a Cayley graph and the graph $\bar{\mathcal{G}} ( G,H,S_{O} )$
is the graph with vertices $G$ and edges
\[
\begin{dcases}
  ( h,h s )      & \forall  h  \in H,\hspace{2mm} \forall s  \in S_{O} , \\
  ( x,x s^{-1} )  & \forall x  \in \bigcup_{s  \in  S_{O}} H s, \hspace{2mm}  \forall s
  \in  H x  \cap  S_{O} .
\end{dcases}
\]
According to the alternative definition, the group-subgroup pair graph
$\mathcal{G} ( G,H,S )$ can be thought informally as the undirected graph
consisting of an inner Cayley graph of $H$ determined by $S_{H}$ and an outer
graph that connects vertices of $H$ with vertices of $\bigcup_{s  \in  S_{O}}
H s $. The notation $S_{O}$ and $S_{H}$ is used throughout this paper with the
same meaning as in the previous definition.

When $G=H$, the generating set is $S=S \cap H$ and the corresponding graphs
is a Cayley graph, in other words, $\mathcal{G} ( G,H,S ) =\mathcal{G} ( G,S
)$. This justifies the claim that the group-subgroup pair graphs are a
generalization of the Cayley graphs and the seemingly artificial condition of
symmetry of $S \cap H$. When $S$ is empty, we call the resulting pair graph
$\mathcal{G} ( G,H,S )$ or Cayley graph $\mathcal{G} ( G,S )$ trivial.

\begin{example}
  \label{example1}Let $G = \mathbbm{Z}/12\mathbbm{Z}$, $H= \{ 0,3,6,9 \}$, and
  $S= \{ 2,4,5,7,8 \}$, then $S_{H} =  \emptyset$ and $S_{O} =S$. The
  corresponding pair-graph $\mathcal{G} ( G,H,S )$ can be seen on figure
  \ref{graph1}.
\end{example}

\begin{figure}[h]
  \resizebox{5cm}{5cm}{\includegraphics{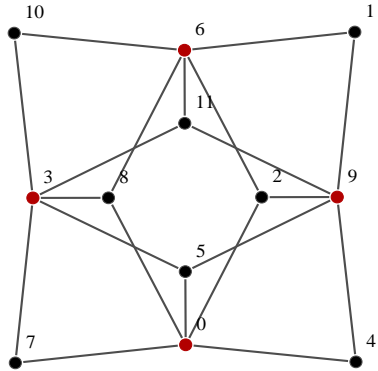}}
  \caption{The pair graph $\mathcal{G} ( \mathbbm{Z}/12\mathbbm{Z} ,
    H ,S )$.}
  \label{graph1}
\end{figure}

\begin{example}
  \label{finitefield}Let $K=\mathbbm{F}_{7^{2}}$ , and the prime field
  $H=\mathbbm{F}_{7}$ of $K$ considered as a subgroup of the additive group
  $K$, then $K$ is the direct sum of seven copies of $H$. Let $\varphi$ be the
  norm map of $K$ as a field extension of $H$, then if $S= \varphi^{-1} ( \{
  5,6 \} )$, we obtain the following pair-graph $\mathcal{G} ( K,H,S )$

  \begin{figure}[h]
    \resizebox{5cm}{5cm}{\includegraphics{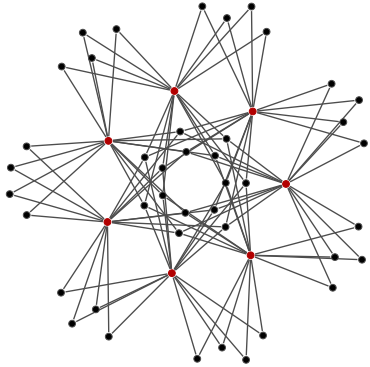}}
    \caption{The pair-graph $\mathcal{G} ( \mathbbm{F}_{7^2} ,
      \mathbbm{F}_7 ,S )$.}
  \end{figure}

  This graph consists of vertices of degree 2,4 and 16.
\end{example}

\begin{example}
  \label{random}Matrix groups over finite fields have been used in the
  construction of families of Ramanujan graphs. Set $G= \tmop{GL}_{2} (
  \mathbbm{F}_{5} )$, where $\mathbbm{F}_{5}$ is the finite field of $5$
  elements and $H= \tmop{SL}_{2} ( \mathbbm{F}_{5} )$, then for a random
  subset $S$ of 7 elements taken from the complement of $H$ in $G$, we obtain
  the following pair-graph $\mathcal{G} ( G,H,S )$.

  \begin{figure}[h]
    \resizebox{7.5cm}{7.5cm}{\includegraphics{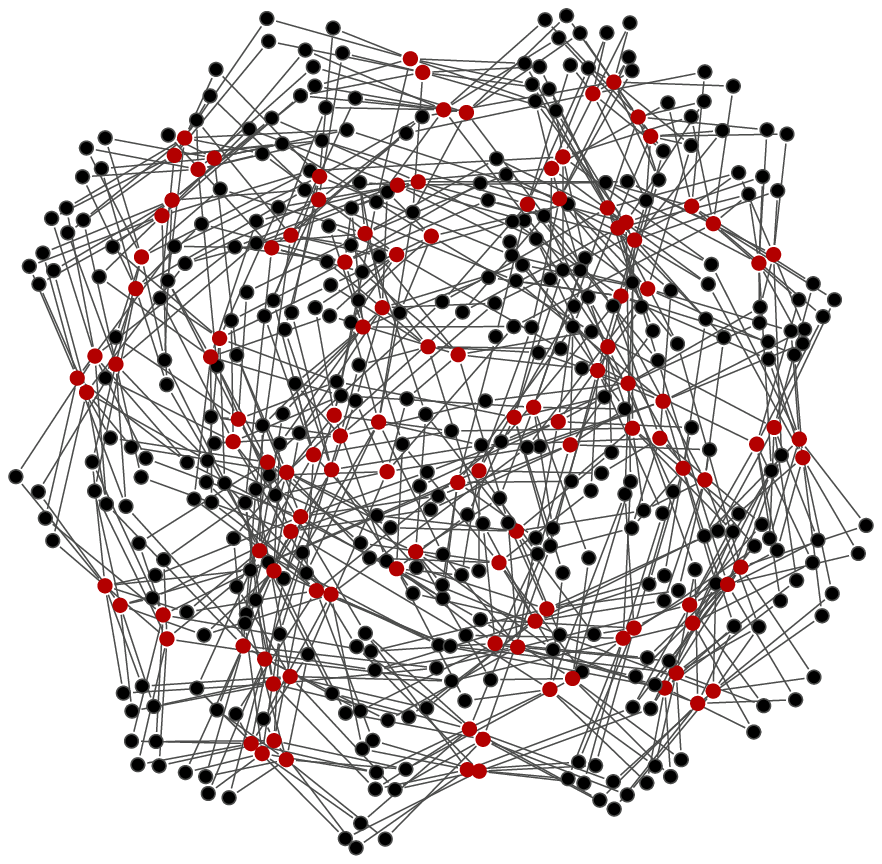}}
    \caption{A pair-graph $\mathcal{G} ( \tmop{GL}_{2} ( \mathbbm{F}_{5} ) ,
      \tmop{SL}_{2} ( \mathbbm{F}_{5} ) ,S )$.}
  \end{figure}

  This connected graph has 480 vertices with vertices of degrees $2,3$ and
  $7$.
\end{example}

Note that none of the graphs of the preceding examples can be constructed as
Cayley graphs.

\begin{remark}
  We briefly mention another generalization of Cayley graphs for group $G$ and
  subgroup $H$, the \tmtextit{Schreier Coset Graph}. For a symmetric subset
  $S$ of $G$, the Schreier coset graph is defined as the graph with the set
  $G/H$ of coset as vertices and where two cosets $H x$ and $H y$ are adjacent
  when there is an $s \in S$ such that
  \[ H x s = H y. \]
  A Schreier Coset graph can have multiples edges and loops (even when $e \nin
  S$). This kind of graphs have been used for coset enumeration techniques. A
  detailed exposition can be found in {\cite{Coxeter1972}}.

  For a given group $G$, the Schreier Coset graph is a Cayley graph when $H=
  \{ e \}$, whereas the group-subgroup pair graph is a Cayley graph when
  $H=G$.
\end{remark}

\subsection{Relation with group matrices}

As stated in the Introduction, one of the motivation for the group-subgroup
pair graph comes from the extension of the group determinant for a
group-subgroup pair, called {\tmem{wreath determinant for group-subgroup
    pair}}s. In this subsection we show how one can relate the adjacency matrix of
a Cayley graph with the group matrix of the corresponding group; then, by
extending the idea for the matrix used for the wreath determinant for
group-subgroup pairs we obtain the rows corresponding to the subgroup on the
adjacency matrix of a certain group-subgroup pair graph, which is enough to
determine the complete adjacency graph.

For a group $G= \{ g_{1} , \ldots ,g_{n} \}$, consider a polynomial ring $R$
containing the indeterminates $x_{g_{i}}$, for $g_{i}   \in G$, then the
\tmtextit{group matrix} is a matrix $\mathcal{M} ( G, \phi )$ in
$\tmop{Mat}_{n,n} ( R )$ defined by
\[ \mathcal{M} ( G, \phi )_{i,j} =x_{g_{i}^{-1} g_{j}}  \]
for $i,j \in [ n ]$ and where $\phi :G \rightarrow [ n ]$ is an enumeration
function for $G$, used implicitly. The determinant $\Theta ( G )$ of the group
matrix is called \tmtextit{group determinant} of $G$ and does not depend on
the chosen enumeration of the elements of $G$. For $i,j$ we have $( g_{i}^{-1}
g_{j} )^{-1} =g^{-1}_{j} g_{i}$, therefore for any element $x_{g}$, the
corresponding transpose element is $x_{g^{-1}}$.

Similarly, for a group $G$ of order $k n$, and subgroup $H$ of order $n$ one
can construct the matrix $\mathcal{M} ( G,H, \phi , \tau ) \in \tmop{Mat}_{n,k
  n} ( R )$ by
\[ \mathcal{M} ( G,H, \phi , \tau )_{i,j} =x_{h_{i}^{-1} g_{j}} , \]
for $h_{i} \in H$, $g_{j} \in G$, $i \in [ n ] , j \in [ n k ]  $and where
$\phi :G \rightarrow [ n k ]$ and \ $\tau :H \rightarrow [ n ]$ are
enumerations functions for $G$ and $H$. Considering only the columns
corresponding to elements of $H$ of the matrix $\mathcal{M} ( G,H, \phi , \tau
)$ one obtains the group matrix of $H$ with respect to the ordering $\tau$.

For a matrix $A \in M_{n,k n}$, the wreath determinant of $A$ is defined as
\[ \tmop{wrdet}_{k} ( A ) = \tmop{det}^{- \frac{1}{k}} ( A_{[ k ]} ) , \]
where $A_{[ k ]}$ is the row $k$-flexing of the matrix $A$ and
$\det^{\alpha}$is the $\alpha$-determinant, for an extensive exposition of the
wreath determinant and its properties the reader is referred to
{\cite{Kimoto2008}}. In the paper {\cite{Hamamoto2014}}, the authors define
the wreath determinant for the pair $G$ and $H$ by
\[ \Theta ( G,H, \phi , \tau ) = \tmop{wrdet}_{k} ( \mathcal{M} ( G,H, \phi ,
\tau ) ) . \]
In contrast with the ordinary group determinant, this wreath determinant for
$G$ and $H$ depends on the enumeration functions $\phi$ and $\tau$.

For a given group $G$ and symmetric subset $S$, by evaluating the
corresponding group matrix $\mathcal{M} ( G, \phi )$ by the rule $x_{s} =1$
for $s \in S$ and $x_{g} =0$ for $g \nin S$ one obtains a symmetric matrix.
Since $g_{i}^{-1} g_{j} =s$ implies $g_{i} s= g_{j}$, the corresponding matrix
is the adjacency matrix of the Cayley graph $\mathcal{G} ( G,S )$.

\begin{example}
  Consider $\mathfrak{S}_{3}$, the symmetric group on three letters with the
  ordering $\phi$ given by $\mathfrak{S}_{3} = \{ e, ( 2,3 ) , ( 1,2 ) , (
  1,2,3 ) , ( 1,3,2 ) , ( 1,3 ) \}$, the group matrix is
  \[ \mathcal{M} ( \mathfrak{S}_{3} , \phi ) = \left(\begin{array}{cccccc}
      x_{1} & x_{2} & x_{3} & x_{4} & x_{5} & x_{6}\\
      x_{2} & x_{1} & x_{4} & x_{3} & x_{6} & x_{5}\\
      x_{3} & x_{5} & x_{1} & x_{6} & x_{2} & x_{4}\\
      x_{5} & x_{3} & x_{6} & x_{1} & x_{4} & x_{2}\\
      x_{4} & x_{6} & x_{2} & x_{5} & x_{1} & x_{3}\\
      x_{6} & x_{4} & x_{5} & x_{2} & x_{3} & x_{1}
    \end{array}\right) , \]
  where $x_{i}$ stands for $x_{g_{i}}$. Considering $S= \{ ( 1,2 ) , ( 1,2,3 )
  , ( 1,3,2 ) \}$ and evaluating in the way mentioned before we get
  \[ A= \left(\begin{array}{cccccc}
      0 & 0 & 1 & 1 & 1 & 0\\
      0 & 0 & 1 & 1 & 0 & 1\\
      1 & 1 & 0 & 0 & 0 & 1\\
      1 & 1 & 0 & 0 & 1 & 0\\
      1 & 0 & 0 & 1 & 0 & 1\\
      0 & 1 & 1 & 0 & 1 & 0
    \end{array}\right) , \]
  which can be verified to be the adjacency matrix of the Cayley graph
  $\mathcal{G} ( \mathfrak{S}_{3} ,S )$.

  \begin{figure}[h]
    \resizebox{5cm}{3.3cm}{\includegraphics{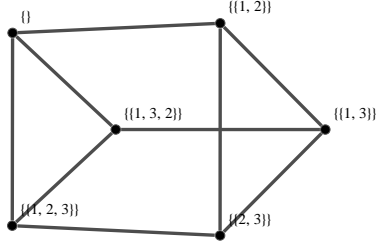}}
    \caption{The Cayley graph $\mathcal{G} ( \mathfrak{S}_{3} ,S )$.}
  \end{figure}
\end{example}

Likewise, for group $G$, subgroup $H$ and subset $S$ as in definition
\ref{def}, by evaluating the group-subgroup matrix with the rule $x_{s} =1$
for $s \in S$ and $x_{g} =0$ otherwise, we obtain a matrix with nonzero
entries $( i,j )$ when $h_{i}^{-1} g_{j} =s \in S$. In other words, there are
ones in the matrix exactly when $h_{i} s=g_{j}$, which is the relation for the
edges of the group-subgroup pair graph $\mathcal{G} ( G,H,S )$ in definition
$\ref{def}$. The resulting matrix corresponds to the rows associated with the
elements of $H$ in the adjacency matrix of the pair-graph $\mathcal{G} ( G,H,S
)$ and can be completed by symmetry to obtain the complete adjacency matrix.

\begin{example}
  Let $G = \mathbbm{Z}/12\mathbbm{Z}$, $H= \{ 0,3,6,9 \}$, and $S= \{
  2,4,5,7,8 \}$ as in example \ref{example1}, the corresponding matrix with
  respect to natural orderings $\phi$ and $\tau$ is
  \[ \mathcal{M} ( \mathbbm{Z}/12\mathbbm{Z},H, \phi , \tau ) =
  \left(\begin{array}{cccccccccccc}
      x_{0} & x_{1} & x_{2} & x_{3} & x_{4} & x_{5} & x_{6} & x_{7} & x_{8} &
      x_{9} & x_{10} & x_{11}\\
      x_{9} & x_{10} & x_{11} & x_{0} & x_{1} & x_{2} & x_{3} & x_{4} & x_{5}
      & x_{6} & x_{7} & x_{8}\\
      x_{6} & x_{7} & x_{8} & x_{9} & x_{10} & x_{11} & x_{0} & x_{1} & x_{2}
      & x_{3} & x_{4} & x_{5}\\
      x_{3} & x_{4} & x_{5} & x_{6} & x_{7} & x_{8} & x_{9} & x_{10} & x_{11}
      & x_{0} & x_{1} & x_{2}
    \end{array}\right) . \]
  Then, evaluating as describe before
  \[ A= \left(\begin{array}{cccccccccccc}
      0 & 0 & 1 & 0 & 1 & 1 & 0 & 1 & 1 & 0 & 0 & 0\\
      0 & 0 & 0 & 0 & 0 & 1 & 0 & 1 & 1 & 0 & 1 & 1\\
      0 & 1 & 1 & 0 & 0 & 0 & 0 & 0 & 1 & 0 & 1 & 1\\
      0 & 1 & 1 & 0 & 1 & 1 & 0 & 0 & 0 & 0 & 0 & 1
    \end{array}\right) , \]
  which can be verified to correspond to the rows of the vertices of $H$ of
  the adjacency matrix of the pair-graph $\mathcal{G} (
  \mathbbm{Z}/12\mathbbm{Z},H,S )$ of example \ref{example1}.
\end{example}

Note that the group matrix can also be defined as $( g_{i} g_{j}^{-1} )$.
Using this definition the resulting Cayley graph is defined by left
multiplication, the same is true for the group-subgroup matrix for the wreath
determinant and the group-subgroup pair graph.

\subsection{Basic properties of pair graphs $\mathcal{G} ( G,H,S )$ }

An isolated vertex is one that is not connected to any other vertices. In
contrast with nontrivial Cayley graphs, group-subgroup pair graphs may contain
isolated vertices even when the generating subset is non empty. The following
result characterizes the presence of isolated vertices in group-subgroup pair
graphs.

\begin{proposition}
  \label{isolated}i) The pair-graph $\mathcal{G} ( G,H,S )$ contains no
  isolated vertices if and only if $S$ contains a representative for each
  coset of $H$ on $G$ different from $H e =H$.

  ii) The vertices $H$ are isolated in $\mathcal{G} ( G,H,S )$ if and only if
  $S$ is the empty set.
\end{proposition}

\begin{proof}
  Suppose $S$ is not empty and contains a representative for each coset, then
  take $x  \in  G-H$, and $s \in  S$ the representative of $H x$, then there
  is $h \in  H$ such that $h s = x$, and therefore $x$ is connected to $h$.
  Conversely, if there are no isolated vertices, by the definition we must
  have $\bigcup_{s  \in  S_{O}} H s = G$. \ The second statement follows
  directly from the definition.
\end{proof}

\begin{example}
  Consider any group $G$ of order $n$ and $H= \{ e \}$, then the graph
  $\mathcal{G} ( G,H,S )$ with $S=G-H$ has no isolated vertices. In fact,
  $\mathcal{G} ( G,H,S )$ is a $T_{n-1}$ {\tmem{star graph}}.
\end{example}

\begin{figure}[h]
  \resizebox{3cm}{3cm}{\includegraphics{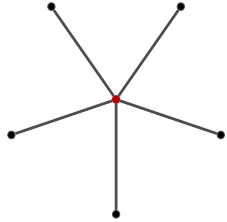}}
  \caption{A $\mathcal{G} ( G, \{ e \} ,S )$ graph with $| G | =6$ and $S=G-
    \{ e \}$.}
\end{figure}

The second part of the proposition shows the relation with Cayley graphs, as a
Cayley graph is empty if and only if $S$ is the empty set.

A graph in which all the vertices have the same degree is called a
\tmtextit{regular graph}, more precisely, if the vertices have degree $k$, the
graph is called a $k$-\tmtextit{regular graph.} An important property of a
Cayley graph $\mathcal{G} ( G,S )$ is that it is $| S |$-regular. Example
\ref{example1} shows that this is not true in general for $\mathcal{G} ( G,H,S
)$, but there is still uniformity on the degree of the vertices in each coset.

\begin{proposition}
  \label{coset}In a pair-graph $\mathcal{G} ( G,H,S )$, all the vertices in
  the same coset have the same degree. \label{Cor1}Namely, the vertices in $H$
  have degree $| S |$ and for $x  \nin  H$ the degree of the vertices in the
  coset $H x$ is $| S \cap H x |$.
\end{proposition}

\begin{proof}
  It is clear by the definition of the outer graph $\bar{\mathcal{G}} (
  G,H,S_{O} )$ that any vertices $x,y  \in  G-H$ in the same coset $H x$ have
  the same degree $| H x  \cap  S_{O} |  =  | H y  \cap  S_{O} |$. The
  vertices in $H$ have degree $| S |$ by construction.
\end{proof}

Returning to example \ref{example1}, $( H+1 ) \cap S= \{ 4,7 \}$, $( H+2 )
\cap S= \{ 2,5,8 \}$, and the cardinality of these sets corresponds to the
degree of the vertex in the respective cosets. We call those types of graphs,
\tmtextit{multi-regular} or more precisely $p_{1} ,p_{2} , \ldots
,p_{r}$-regular graphs, where $p_{i}$ is the degree of the vertices on a given
partition. The preceding discussion suggests that the structure of $G/H$, the
set of cosets of $H$ on G, is closely related to the structure of the graph.

\begin{corollary}
  Let $G$ be a group and $H$ a subgroup of index $[ G:H ] =k+1$. For a \
  subset $S \subset G$ with $S_{H}$ symmetric, consider the pair-graph
  $\mathcal{G} ( G,H,S )$. If $x_{1} , \ldots ,x_{k}$ is a set of
  representatives of the cosets of H incongruent to $e$, then for $h \in H$,
  \[ \deg ( h ) = | S | \geqslant    \sum_{i=1}^{k} | H x_{i}   \cap  S_{O} |
  = | S_{O} | = \sum_{i=1}^{k} \deg ( x_{i} ) , \]
  with equality only when $S_{H} = \emptyset$. In particular,  a nontrivial
  $\mathcal{G} ( G,H,S )$ is regular if and only if $S_{H} =  \emptyset$ and
  $[ G:H ] =2$, or $[ G:H ] =1$.
\end{corollary}

\begin{proof}
  The inequality follows from Proposition \ref{Cor1} and the fact that $| S |
  \geqslant   | S_{O} |$, and the equality happens only when $S_{H} =S \cap
  H= \emptyset$. The \tmtextit{if} part of the proof follows directly from the
  inequality and the definitions. For the \tmtextit{only if} part, consider a
  $j$-regular $\mathcal{G} ( G,H,S )$ graph, then by the preceding proposition
  $| S | =j$ and $| H x \cap S_{O} | =j$ for $x \nin H$. It follows from the
  inequality above that $| S_{O} |  = k j=k | S |$ and therefore $k$ is
  necessarily $0$ or $1$. If $[ G:H ] =2$, then $| S_{O} |  =  | S |$, so $S =
  S_{O}$ and the case $[ G:H ] =1$ gives a Cayley graph.
\end{proof}

Note that in view of Proposition \ref{coset}, we can consider $\deg ( H x )$
as the degree of any of the elements of the coset. In that case, the above
identity can be written as
\[ \deg ( H )   \geqslant   \sum_{i} \deg ( H x_{i} ) , \]
with equality happening only when $S_{H}$ is empty.

\section{General properties of pair-graphs $\mathcal{G} ( G,H,S )$ }

\subsection{Connectedness and connected components}

In this section we consider the connectedness for pair graphs $\mathcal{G} (
G,H,S )$. Recall that a Cayley graph $\mathcal{G} ( G,S )$ is connected if and
only if $\langle S \rangle =G$. We begin by considering the case $S_{H} =S
\cap H$ empty, in other words, none of the vertices of $H$ are adjacent in
$\mathcal{G} ( G,H,S )$.

\begin{lemma}
  \label{conectbip}Let $G \nocomma  $ be a group, H a subgroup and \ $S
  \subset  G$ subset with $S_{H}$ symmetric. If $S_{H} = \emptyset$ then in
  the pair-graph $\mathcal{G} ( G,H,S )$ the vertices of $H$ are in the same
  connected component if and only if $\langle H \cap ( S S^{-1} ) \rangle  $ =
  H.
\end{lemma}

\begin{proof}
  If $\langle H \cap ( S S^{-1} ) \rangle =H$, then it suffices to prove that
  the identity $e$ is connected to an arbitrary $h \in H$. For $h \in H$, we
  have $h = s_{1} s_{2}^{-1} \ldots s_{n-1} s_{n}^{-1}$ with $s_{i}
  s_{i+1}^{-1} \in H$, then if we set $h_{1} =s_{1} s_{2}^{-1} \ldots s_{n-3}
  s_{n-2}^{-1}$, $h_{1}$ is adjacent to $h_{1} s_{n-1} =x_{1}$ and $h$ is
  adjacent to $h s_{n}  = h_{1} s_{n-1} =x_{1}$ so $h_{1}$ is connected to
  $h$. By repeating this process we conclude that $e$ is connected to $h$.

  On the other hand, if in the graph $\mathcal{G} ( G,H,S )$, all the vertices
  of $H$ are in the same connected component, $h \in H$ is connected to $e \in
  H$. Since there are no direct connections between two elements of $H$ or
  $G-H$, there must be path from $e$ to $h$ where every even vertex is is an
  element of $H$, so we have a sequence $h_{0} =e,h_{1} , \ldots ,h_{n-1}
  ,h_{n} =h$ of elements of $H$, where $h_{i}$ and $h_{i+1}$ is connected to
  $x_{i}$ $i=0,2, \ldots ,n-1$ with $x_{i} \in G-H$. That is, we have a
  sequence of edges $( h_{0} ,x_{0} ) , ( x_{0} ,h_{1} ) \ldots ( h_{n-1}
  ,x_{n-1} ) , ( x_{n-1} ,h_{n} )$, as shown in the following diagram.

  \begin{figure}[h]
    \includegraphics{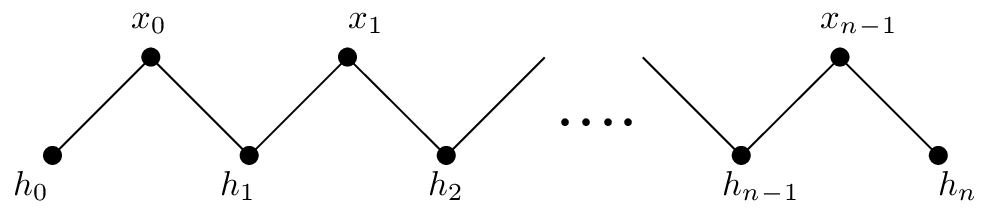}
    \caption{The path from $h_0$ to $h_n$.}
  \end{figure}

  Then, for $s_{i}   \in S,$
  \begin{eqnarray*}
    x_{0} =h_{0} s_{0} & , & x_{0} =h_{1} s_{1}\\
    x_{1} =h_{1} s_{2} & , & x_{1} =h_{2} s_{3}\\
    \vdots  & & \vdots\\
    x_{n-1} =h_{n-1} s_{2n-2} & , & x_{n-1} =h_{n-1} s_{2n-1}
  \end{eqnarray*}
  thus,
  \begin{eqnarray*}
    s_{0} =h_{0} s_{0} =h_{1} s_{1} &  & h_{1} =s_{0} s_{1}^{-1}\\
    h_{1} s_{2} =h_{2} s_{3} &  & h_{2} =h_{1} s_{2} s_{3}^{-1}\\
    \vdots & \Rightarrow & \vdots\\
    h_{n-1} s_{n-2} =h_{n} s_{2n-1} &  & h =h_{n-1} s_{n-2} s_{n-1}^{-1} ,
  \end{eqnarray*}
  it follows that $s_{i} s_{i+1^{}}^{-1} \in H$ and $h \in \langle H \cap S
  S^{^{-1}} \rangle$.
\end{proof}

Note that since a group-subgroup pair graph may contain isolated vertices, the
condition of the lemma alone is not sufficient for connectedness.

\begin{proposition}
  \label{connectbip}With the same notation as before, if $S_{H} = \emptyset$
  then the pair-graph $\mathcal{G} ( G,H,S )$ is connected if and only if
  $\langle H \cap S S^{-1} \rangle =H$ and S contains representatives of all
  the cosets of $H $different from $H$.
\end{proposition}

\begin{proof}
  The result follows from part follows from Lemma \ref{conectbip}, Proposition
  \ref{isolated} and the observation that any vertex $x \in G-H$ must be
  connected to some $h \in H$ which is in turn connected to the identity $e
  \in H$.
\end{proof}

\begin{theorem}
  \label{connectness}A pair-graph $\mathcal{G} ( G,H,S )  $is connected if and
  only if
  \[ \langle H \cap ( S_{H} \cup S_{o} S_{o}^{-1} ) \rangle = H \]
  and $S$ contains representatives of all the cosets of $H$ different
  from the coset $H$.
\end{theorem}

\begin{proof}
  First we see that the vertices of $H$ are in the same connected component if
  and only if $\langle H \cap ( S_{H} \cup S_{o} S_{o}^{-1} ) \rangle$=H. The
  proof of this fact is the same as that of Lemma \ref{conectbip} while
  considering that in the path from a $e  \in H$ to $h \in H$ there may be
  edges connecting elements $h_{1} ,h_{2}$ from $H$, in such case we have
  $h_{2} =h_{1} s_{H}$, with $s_{H} \in S_{H}$. Then the result follows like
  in Proposition \ref{connectbip}.
\end{proof}

\begin{example}
  \label{example2}Let $G=\mathbbm{Z}/12\mathbbm{Z}$, $H \cong
  \mathbbm{Z}/4\mathbbm{Z}$. Set $S_{1} = \{ 1,7 \}$ and $S_{2} = \{
  4,5,6,10,11 \}$, the corresponding group-subgroup pair graphs are the
  following.

  \begin{figure}[h]
    \resizebox{3cm}{5cm}{\includegraphics{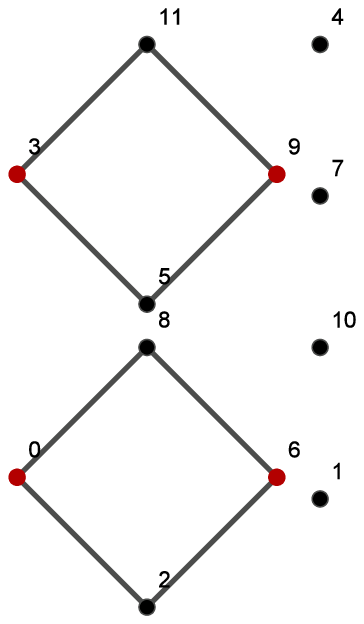}} \ \ \ \
    \ \ \ \ \ \ \ \ \
    \resizebox{3cm}{5cm}{\includegraphics{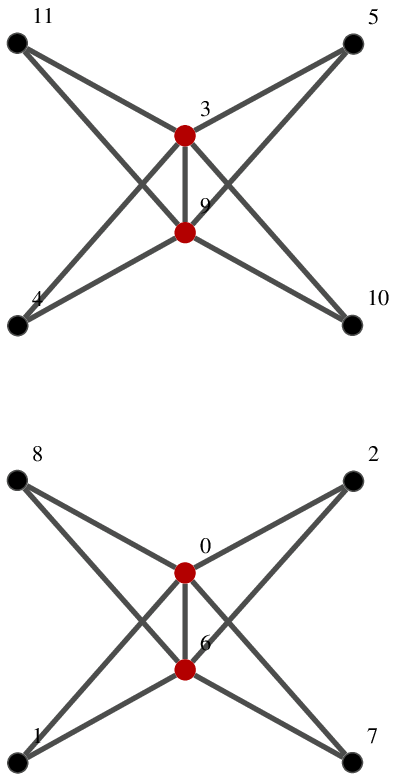}}
    \caption{The pair-graphs $\mathcal{G} ( G,H,S_{1} )$ and $\mathcal{G} ( G,H,S_{2} )$.}
  \end{figure}

  Note that $\langle H \cap S_{1} S_{1} \rangle^{-1} = \{ 0,6 \}$ \ and
  $\langle H \cap ( S_{H_{2}} \cup S_{O_{2}} S_{O_{2}}^{-1} ) \rangle = \{ 0,6
  \}$, so neither graph is connected, as can be seen in the diagrams. On the
  other hand, as there are no elements of the coset $H+5= \{ 5,8,11,2 \}$ on
  $S_{1}$, all the vertices of that coset are isolated on $\mathcal{G} (
  G,H,S_{1} )$.
\end{example}

If a graph is not connected, the characterization of the connected components
of the graph is desirable. For Cayley graphs, the connected component of the
identity is the subgroup $\langle S \rangle$, and each of the cosets in $G$
are the connected components of the graph. This identification is not
identical for group-subgroup pair graphs as the connected component of the
identity may include vertices from $G-H$. In particular, the subgroup $\langle
H \cap ( S_{H} \cup S_{o} S_{o}^{-1} ) \rangle$ of $H$ contains the elements
of $H$ that are in the connected component of the identity and the cosets of
this subgroup are the intersection of $H$ with certain connected components of
the graph.

\begin{proposition}
  \label{comps}With the same notation as before, let $U= \langle H \cap (
  S_{H} \cup S_{O} S_{O}^{-1} ) \rangle$, then the identity component
  $~ \Gamma_{e} $ of $~ \mathcal{G}(G,H,S) $ consists of the vertices $~ U \cup
  \left( \bigcup_{s \in S_{O}} U s \right)$. The remaining connected
  components of the pair-graph $\mathcal{G} ( G,H,S )$ are either of the type
  $\Gamma_{h} =h \Gamma_{e}$ for $h \in H$ or the type $\{ x \}$ for $x \in
  G-H$.
\end{proposition}

\begin{proof}
  The first statement follows from the preceding discussion and the definition
  of the pair-graph $\mathcal{G} ( G,H,S )$. Any path $e,g_{1} ,g_{2} , \ldots
  ,g_{n}$ from the identity $e$ to $g_{n}$ corresponds uniquely to a path $h,h
  g_{1} , \ldots ,h g_{n}$ from $h$ to $h g_{n}$ so the connected component of
  $h \in H$ is $\Gamma_{h} =h \Gamma_{e}$. For $x \in G-H$ if $x$ is an
  isolated vertex its connected component is $\{ x \}$, otherwise it is
  connected to an $h \in H$ so its connected component is of the type
  $\Gamma_{h}$.
\end{proof}

A consequence of the above proposition is that an arbitrary connected
component $\Gamma$ of $\mathcal{G} ( G,H,S )$ has cardinality equal to
$\Gamma_{e}$ or 1. Moreover, for first case, we also have $| \Gamma \cap H | =
| \Gamma_{e} \cap H |$ and $| \Gamma -H | = | \Gamma_{e} -H |$.

For Cayley graphs, the number of connected components of the graph is the
index $[ G: \langle S \rangle ]$. The existence of isolated vertices even for
non empty generating sets makes the situation for pair-graphs slightly more
complicated.

\begin{theorem}
  \label{components}The number of connected components of $\mathcal{G} ( G,H,S
  )$ is
  \[ [ H: \langle H \cap ( S_{H} \cup S_{o} S_{o}^{-1} ) \rangle ] \noplus
  \noplus +  \vert G-H \vert  -  \left\vert \bigcup_{s \in S_{O}} H s \right\vert . \]
\end{theorem}

\begin{proof}
  By the preceding proposition, the first term in the formula is the number of
  connected components $\Gamma_{h}$ that occur on $H$, the second and third
  terms count the number of isolated points in $G-H$, by Proposition
  \ref{isolated}. Since there are not connections between elements of $G-H$,
  this is the number of connected components of the graph.
\end{proof}

Proposition \ref{comps} and Theorem \ref{components} completely characterize
the connected components for the pair-graphs $\mathcal{G} ( G,H,S )$ for given
group $G$, subgroup H and valid subset $S \subset G$.

\begin{example}
  For the pair-graph of example \ref{example1} we have $S= \{ 2,4,5,7,8 \}$,
  since $h=3=8-5  \in  S S^{-1}$, $H$ is generated by $S S^{-1}$ and the first
  term is 1, the second term is 8 and since all the cosets are represented the
  last term is 8, and we get 1 connected component.

  For the pair-graphs generated by $S_{1}$ and $S_{2}$ of example
  \ref{example2}, in both cases the first term is 2, the next term is $8$, and
  the final term is $4$ for the graph generated by $S_{1}$ and $8$ for the
  graph generated by $S_{2}$, therefore $\mathcal{G} ( G,H,S_{1} )$ has 6
  connected components and $\mathcal{G} ( G,H,S_{2} )$ has 2 connected
  components as confirmed visually in the diagrams.
\end{example}

\subsection{\label{groupactions}Vertex transitivity and group actions}

A graph is \tmtextit{vertex transitive} when for any pair of different
vertices $x$ and $y$ there is graph automorphism $\varphi$ such that $\varphi
( x ) =y$. Cayley graphs are naturally vertex transitive by means of left
translations $L_{g}$ with $g \in G$. Any vertex transitive graph must be
regular, therefore by Proposition \ref{coset}, we have the following result.

\begin{proposition}
  Nontrivial pair-graphs $\mathcal{G} ( G,H,S )$ are not vertex transitive
  when $[ G:H ] \geqslant 3$.
\end{proposition}

However, we still see restricted transitivity when considering left
translations by elements of $H$ on $\mathcal{G} ( G,H,S )$.

\begin{proposition}
  \label{leftaction}The left action of H in $\mathcal{G} ( G,H,S )$ is a graph
  automorphism. Moreover, for any $h_{1} ,h_{2} \in H$, the action is
  transitive. The same holds for any coset $H$x, $x \in G$.
\end{proposition}

\begin{proof}
  For $h \in H$ and $s \in S$, \ the edge $( h,h s )$ is in $\mathcal{G} (
  G,H,S )$. Now, for $h' \in H$, the images of $h$, and $h s$ under the left
  action $L_{h'}$ are $h' h$ and $h' h s$, which are adjacent in $\mathcal{G}
  ( G,H,S )$. The map is a bijection, and has inverse $L_{h'^{-1}}$ so it is a
  graph automorphism. For $h_{1} ,h_{2} \in H$, consider $h' =h_{1}
  h_{2}^{-1}$, then $L_{h'} ( h_{1} ) =h_{2}$ as required. For $x_{1} ,x_{2}
  \in H x $, take $h' \in H$ such that $L_{h'} ( x_{1} ) =x_{2}$.
\end{proof}

Note that the left action of arbitrary $g \in G$ is not necessarily a graph
automorphism. For instance, in example \ref{graph1} the action of $g=1$ is not
a graph endomorphism, as the image of $0$ is $1$, and $\tmop{degree} (0)=5$,
but $\tmop{degree} ( 1 ) =2$.

Let $\mathcal{L} ( G )$ be the set of complex valued functions defined on $G$,
and consider $\lambda_{H} :H \rightarrow \tmop{GL} ( \mathcal{L} ( G ) )$ the
left regular representation of $H$ given by the action $\lambda_{H} ( h ) f (
x ) =f ( h^{-1} x )$ for h$\in H$ and $x \in G$. Recall that for any graph
with vertices on $G$ we can associate the adjacency operator $A$, acting on
$\mathcal{L} ( G )$ in the following way,
\[ A f ( x )  =  \sum_{x \sim y} f ( y ) . \]
\begin{proposition}
  \label{leftreg}Let $A$ be the adjacency matrix for a group-subgroup pair
  graph $\mathcal{G} ( G,H,S )$, then for any $h \in H$
  \[ \lambda_{H} ( h ) A=A \lambda_{H} ( h ) . \]
\end{proposition}

\begin{proof}
  Note that for the pair-graph $\mathcal{G} ( G,H,S )$ and $f \in \mathfrak{L}
  ( G )$, the adjacency operator is given by

  \[
  A f(y) =
  \begin{cases}
    \sum_{s \in S} f ( y s )  & \text{if }  y \in H  \\
    \sum_{s \in S \cap H y} f ( y s^{-1} )      & \text{if } y \in G-H
  \end{cases}
  \]

  The result then follows from direct calculation by considering that for any
  $x \in G$, $x$ and $h^{-1} x$ are in the same coset, in other words, $H x=H
  h^{-1} x$ for any $h \in H$.
\end{proof}

For $h \in H$, $\lambda_{H} ( h )$ is a permutation matrix for the elements of
$G$ corresponding to the left multiplication by $h$. By Theorem 15.2 of
{\cite{Biggs1996}}, a bijection $\varphi$ of the vertices of a graph is a
graph automorphism if $M_{\varphi} A=A M_{\varphi}$, where $M_{\varphi}$ is
the permutation matrix associated with $\varphi$; therefore Proposition
\ref{leftreg} is only a representation theoretic restatement of Proposition
\ref{leftaction}.

\subsection{The trivial eigenvalues of pair-graphs $\mathcal{G} ( G,H,S )$}

The trivial eigenvalue of a $k$-regular graph is $\mu =k$ and it corresponds
to any constant eigenfunction $f$ on the vertices of the graph. By extension,
the trivial eigenvalue of a Cayley graph $\mathcal{G} ( G,S )$ is $\mu =k$. In
this section we extend this notion to the group-subgroup pair graphs.

\begin{theorem}
  \label{trivialeigen}Let $G$ be a group, $H$ a subgroup of $G$ of index $[
  G:H ] =k+1$ with $k \geqslant 1$, and \ $S \subset G$ a nonempty subset with
  $S_{H}$ symmetric and $| S_{O} | \neq 0$. Consider $e=x_{0} ,x_{1} , \ldots
  ,x_{k}$ a set of representatives of the cosets of $H$ in $G$ and set $S_{i}
  =S \cap H x_{i}$ , for $i \in [ k ]$. \ Then
  \[ \mu^{\pm} = \frac{| S_{H} | \pm \sqrt{| S_{H} |^{2} +4 \left(
        \sum_{1}^{k} | S_{i} |^{2}   \right)}}{2} \]
  are eigenvalues of the graph $\mathcal{G} ( G,H,S )$. A corresponding
  eigenfunction is defined by

  \[
  f^{\pm}(y) =
  \begin{cases}
    \mu^{\pm}, & \text{if } y \in H \\
    | S_{i} | & \text{if } y \in H y_{i},\hspace{2mm} i \in [ k ] .
  \end{cases}
  \]

\end{theorem}

\begin{proof}
  Note that the numbers $\mu^{\pm}$ satisfy
  \[ ( \mu^{\pm} )^{2} - | S_{H} | \mu^{\pm} - \sum_{1}^{k} | S_{i} |^{2}
  =0. \]
  Then, for $h \in H$, we have
  \begin{align*}
    A f^{\pm} ( h ) = & \sum_{s \in S} f^{\pm} ( h s )\\
    = & \sum_{s \in S_{H}} f^{\pm} ( h s )  + \sum_{i=1}^{k} \sum_{s \in
      S_{i}} f^{\pm} ( h s )\\
    = & | S_{H} | \mu^{\pm} + \sum_{i=1}^{k} | S_{i} |^{2}  =  ( \mu^{\pm})^{2}\\
    = & \mu^{\pm} f ( h ) .
  \end{align*}
  Similarly, for $x \in H x_{i}$, $i=1, \ldots ,k$,
  \begin{align*}
    A f^{\pm} ( x ) = & \sum_{s \in S_{i}} f ( x s^{-1} )\\
    =  & \mu^{\pm} | S_{i} |  \\
    =  & \mu^{\pm} f^{\pm} ( x ) . \qedhere
  \end{align*}
\end{proof}

Note that for the case $| S_{O} | =0$, $\mu^{+}$ is an eigenvalue with
corresponding eigenvector $f^{+}$ as defined in the above Theorem, but $f^{-}
\equiv 0$ so it is not an eigenfunction.

\begin{proposition}
  With the same notation as before, the eigenvalue $\mu^{+}$ is the largest
  eigenvalue of the graph $\mathcal{G} ( G,H,S )$ with multiplicity $[ H,
  \langle H \cap ( S_{H} \cup S_{O} S_{O}^{-1} ) \rangle ]$.
\end{proposition}

\begin{proof}
  First, consider the case of a connected pair-graph $\mathcal{G} ( G,H,S )$,
  in particular $| S_{i} | \neq 0$ for all $i \in [ k ]$ and the eigenfunction
  $f^{+}$ takes only positive values. By Thm. 8.1.4 and Cor. 8.1.5 of
  {\cite{Knauer2011}}, any eigenfunction that only takes nonzero values of the
  same sign corresponds to the largest eigenvalue, which has multiplicity 1.
  The proposed eigenfunction satisfies this condition, therefore corresponds
  to the largest eigenvalue and it is an eigenvalue of multiplicity one, so
  the statement of the proposition follows.

  For the remaining case, let $h \in H$ and consider the connected component
  $\Gamma_{h}$ as a subgraph of $\mathcal{G} ( G,H,S )$ and note that $f^{+}
  |_{\Gamma_{h}}$ is an eigenfunction of $\Gamma_{h}$ with eigenvalue
  $\mu^{+}$. Then by \ the same argument as before, $\mu^{+}$ is the largest
  eigenvalue of $\Gamma_{h}$ with multiplicity 1. Now, it is well known that
  the characteristic polynomial $p_{\mathcal{G}} ( x )$ of the graph
  $\mathcal{G}=\mathcal{G} ( G,H,S )$ is the product of the characteristic
  polynomials of its connected components,
  \[ p_{\mathcal{G}} ( x ) =p_{\Gamma_{h_{1}}} ( x ) p_{\Gamma_{h_{2}}} ( x )
  \ldots p_{\Gamma_{h_{r}}} ( x ) x^{l} , \]
  where $r$ is the number of connected components of $\mathcal{G} ( G,H,S )$
  containing elements of $H$ and $l$ the number of isolated vertices.
  Moreover, since $\mu^{\pm}$ is the largest root of $p_{\Gamma_{h_{i}}} ( x
  )$ for each $h_{i} \in H$, then is the largest root of $p_{\mathcal{G}} ( x
  )$ and therefore the largest eigenvalue of $\mathcal{G} ( G,H,S )$.
  Furthermore, since $\mu^{+}$ is a simple eigenvalue for each of the
  subgraphs $\Gamma_{h_{i}}$ , then by Theorem \ref{components} $\mu^{+}$ is
  an eigenvalue of $\mathcal{G} ( G,H,S )$ with multiplicity equal to $[ H,
  \langle H \cap ( S_{H} \cup S_{O} S_{O}^{-1} ) \rangle ]$.
\end{proof}

From the decomposition of the characteristic polynomial of the adjacency
matrix, we can obtain a bound for the multiplicity of the eigenvalue \(\mu = 0\).

\begin{proposition}
  With the same notation as before, $\mu =0$ is an eigenvalue of multiplicity at least
  \[ |G| - |H| - \min(\left\vert \bigcup_{s \in S_{O}} H s \right\vert,|H|) \]
  of the nontrivial pair-graph $\mathcal{G} ( G,H,S )$.
\end{proposition}

\begin{proof}
  First, suppose that the pair-graph $\mathcal{G} ( G,H,S )$ contains isolated vertices, then
  from the decomposition of the characteristic polynomial of the adjacency
  matrix we have that $\mu = 0$ is an eigenvalue with multiplicity at least the number of
  isolated components, by {\ref{components}} this number is \( |G| - |H| - \left\vert \bigcup_{s \in S_{O}} H s \right\vert \). For a general pair-graph $\mathcal{G} ( G,H,S )$, consider an eigenfunction $f$ associated with the
  eigenvalue $\mu=0$ with $f(h)=0$ for $h \in H$, then $f$ must satisfy the \( |H| \) linear equations
  \[
      \sum_{s \in S_O} f(h s) = 0.
  \]
  The matrix \(B \) corresponding to this system is a \(|H| \times |G-H| \) matrix, then from
  elementary linear algebra it holds that \( |H| \) is an upper bound for the rank of \(B \). Therefore
  the kernel of \( B \) has dimension at least $|G|-2|H|$, so the multiplicity of the eigenvalue
  \(\mu = 0\) is at least $|G|-2|H|$. The result follows from considering the two cases at
  the same time.

\end{proof}

\begin{example}
  The pair-graph in example \ref{finitefield}, has $| S_{O} | = | S | =16$,
  with four cosents of degree 2 and two cosets of degree 4, therefore the
  trivial eigenvalues are $\pm 4 \sqrt{3}$. The pair-graph in example
  \ref{random} has $| S_{O} | = | S | =7$, with two cosets of degree 2 and one
  coset of degree 3, then $\pm \sqrt{17}$ are the corresponding trivial
  eigenvalues. The pair-graph in example \ref{bipartite} has $| S_{H} | =2$,
  $| S_{1} | = | S_{2} | =2$, therefore the trivial eigenvalues are $\mu^{+}
  =4$ and $\mu^{-} =-2$.
\end{example}

\subsection{Bipartite pair-graphs $\mathcal{G} ( G,H,S )$ }

A graph is \tmtextit{bipartite} when there is a bipartition $V_{+} ,V_{-}$ of
the vertices such that any pair of vertices in the same subset are not
adjacent. For a group $G$ and symmetric subset $S$, if there is an
homomorphism $\chi :G \rightarrow \{ -1,1 \}$, such that $\chi ( S ) = \{ -1
\}$ then the Cayley graph $\mathcal{G} ( G,S )$ is bipartite, this condition
is also necessary when the graph is connected.

\begin{proposition}
  If a group homomorphism $\chi :G \rightarrow \{ -1,1 \}$, such that
  $\chi ( S ) = \{ -1 \}$ exists, then the graph $\mathcal{G} ( G,H,S )$ is
  bipartite. The converse is true if $\mathcal{G} ( G,H,S )$ is connected and
  $S$ is symmetric.
\end{proposition}

The proof of this proposition is the same as the one for Cayley graphs, see
for example chapter 4 of {\cite{Davidoff2003}}. Note that if such a
homomorphism exists then $\chi ( S ) = \chi ( S^{-1} ) = \{ -1 \}$, so the
elements of $S$ and $S^{-1}$ would be on the same partition.

\begin{example}
  \label{bipartite}Let $G=A_{4}$, the alternating group of 4 letters, $H$, the
  Klein four group embedded as a subgroup of $G$, and $S= \{ ( 1,2 ) (
  3,4 )$, $( 1,4 ) ( 2,3 )$, $( 1,2,3 )$, $( 1,4,3 )$, $( 2,3,4 )$, $( 2,4,3 ) \}$, using
  cycle notation. Observe that $( 1,3,2 )$, the inverse of $( 1,2,3 )^{}$, is
  not contained in $S$. The resulting $\mathcal{G} ( G,H,S )$ is a bipartite
  graph.

  \begin{figure}[h]
    \resizebox{7.5cm}{4cm}{\includegraphics{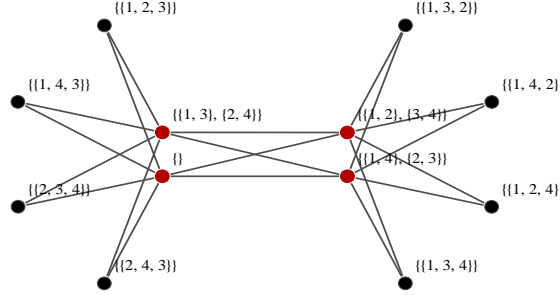}}
    \caption{The bipartite pair-graph $\mathcal{G} ( G,H,S )$.}
  \end{figure}

  Any homomorphism $\chi :G \rightarrow \{ -1,1 \}$ with $\chi ( S ) = \{ -1
  \}$, would have $\chi ( S^{-1} ) = \{ -1 \}$, but in this case $( 1,3,2 ) =
  ( 1,2 ) ( 3,4 ) \cdummy ( 1,4,3 )$, therefore there are no homomorphisms
  that satisfy the conditions.
\end{example}

For the case $S \cap H= \emptyset$, the sets $H$ and $G-H$ are a bipartition
of $\mathcal{G} ( G,H,S )$, so the nontrivial pair-graphs $\mathcal{G} ( G,H,S
)$ with with $S \cap H= \emptyset$ are bipartite. An example can be seen in
Figure \ref{graph1}.

\section{Regular pair-graphs $\mathcal{G} ( G,H,S )$ }

From the results of Section $2$, a group-subgroup pair graph $\mathcal{G} (
G,H,S )$ is regular when $H=G$, or when $H$ is a subgroup or order 2 and
$S_{H}$ is empty. In this section we restrict to the latter case, where the
resulting pair-graph is also bipartite. \

\begin{proposition}
  \label{bipcayley}Let $G$, $H$ and $S$ be as described above. If $S$ is a
  symmetric set, the resulting $\mathcal{G} ( G,H,S )$ is a Cayley graph.
  Namely, $\mathcal{G} ( G,H,S ) =\mathcal{G} ( G,S )$.
\end{proposition}

\begin{proof}
  The conditions imply that $S_{O} =S$, then by the alternative definition of
  the pair-graph the edges are given by
  \begin{equation*}
    ( h,h s ) , \forall  h  \in  H, \hspace{2mm} \forall s  \in  S \hspace{5mm}
    \tmop{ and } \hspace{5mm} ( x,x s^{-1} ) , \forall x  \in   ( G-H ),\hspace{2mm} \forall s  \in  S.
  \end{equation*}
  Since S is symmetric one can simply write $( x,x s )$, $\forall  x  \in  G$,
  $\forall s  \in  S$, which is the definition of Cayley graph.
\end{proof}

Proposition \ref{bipcayley} states another way in which the graph $\mathcal{G}
( G,H,S )$ can be seen as a generalization of certain class of bipartite
Cayley graphs. As an example, for $p$,$q$ prime numbers, when $p$ is not a
square modulo $q$, the $X^{p,q}$ Ramanujan graphs from Lubotzky, Phillips and
Sarnak {\cite{Lubotzky1988}} can be identified with pair-graphs $\mathcal{G} (
\tmop{PGL}_{2} ( q ) , \tmop{PSL}_{2} ( q ) ,S_{p,q} )$ .

\subsection{Generating sets and actions.}

In this section and the next one we fix a group $G$ and subgroup
$H$ of order $n$ and index 2. Let $\mathcal{P}$ denote the power set of $G-H$.
In this section we consider transformations on the generating sets that result
in isomorphic group-subgroup pair graphs.

\begin{proposition}
  \label{rigact}Let $S_{1} \in \mathcal{P}$, if $S_{2}  =R_{h'} ( S_{1} )$,
  where $R_{h'}$ is the right translation by $h' \in H$, then $\mathcal{G} (
  G,H,S_{1} ) \cong  \mathcal{G} ( G,H,S_{2} )$.
\end{proposition}

\begin{proof}
  For $h' \in H$ consider the map $  \varphi_{h'} :G  \rightarrow  G$, defined
  as $\varphi_{h'} ( x ) =x h'$ for $x  \in  G-H$ and $\varphi ( h ) =h$ for
  $h \in H$. The map is clearly bijective, then for $h  \in H  \nocomma ,s \in
  S_{1}$, we have
  \begin{align*}
    ( h,h s ) & \longmapsto  ( h,h s h' )  =  ( h,h s' ),\\
    ( x,x s^{-1} ) & \longmapsto  ( x h' \nocomma \nocomma ,x s^{-1} )  =  (
    x h' \nocomma ,x h' s'^{-1} ) ,
  \end{align*}
  with $s'  =s h' \in S_{2}$, therefore is an graph homomorphism with inverse
  $\varphi_{h'^{-1}}$ and the graphs $\mathcal{G} ( G,H,S_{1} )$ and
  $\mathcal{G} ( G,H,S_{2} )$ are isomorphic.
\end{proof}

\begin{proposition}
  Let $S_{1} \in \mathcal{P}$, if $S_{2}  = \psi ( S_{1} )$, where $\psi$ is a
  group automorphism of G, then the pairs graphs $\mathcal{G} ( G,H,S_{1} )$
  and $ \mathcal{G} ( G,H,S_{2} )$ are isomorphic.
\end{proposition}

\begin{proof}
  Since $[ G:H ] =2$, $H$ is invariant under $\psi$, then for any $h  \in H, x
  \in G-H,s \in S_{1}$,
  \begin{align*}
    ( h ,h s ) & \longmapsto  ( \psi ( h ) , \psi
    ( h ) \psi ( s ) )  = ( \psi ( h ) , \psi ( h ) s' ),\\
    ( x,x s^{-1} ) & \longmapsto  ( \psi ( x ) , \psi ( x ) s'^{-1} ) ,
  \end{align*}
  with $s'  =  \psi ( s ) \in S_{2}$, therefore $\psi$ is a graph isomorphism
  with inverse $\psi^{-1}$.
\end{proof}

Any orbit on $\mathcal{P}$ of right actions $R_{h}$ by elements of $H$ and
group isomorphisms $\psi$ of $G$ consists of sets that generate isomorphic
graphs.

\subsection{Spectra of pair-graphs $\mathcal{G} ( G,H,S )$ }

As mentioned in the beginning of the section, under the current assumptions
the pair-graphs $\mathcal{G} ( G,H,S )$ are regular bipartite. The spectrum of
these graphs is symmetric about 0 and its largest eigenvalue is the trivial
eigenvalue $\mu_{0} = | S |$. Moreover, any eigenvalue $\mu$ satisfies $| \mu
| \leqslant | S |$. In this case the adjacency operator for the
pair-graph $\mathcal{G} ( G,H,S )$ is given by

\[
A f(y) =
\begin{cases}
  \sum_{s \in S} f ( y s )       & \text{if } y \in H \\
  \sum_{s \in S} f ( y s^{-1} )   & \text{if } y \in G-H
\end{cases}
\]

\begin{proposition}
  If $S_{1} \cup S_{2} =G-H$ and $S_{1} \cap S_{2} = \emptyset$, then the
  adjacency operator of $\mathcal{G} ( G,H,G-H )$ is the sum of the adjacency
  operators of $\mathcal{G} ( G,H,S_{1} )$ and $\mathcal{G} ( G,H,S_{2} )$.
\end{proposition}

\begin{proof}
  Follows from the definition
  \begin{align*}
    \sum_{s \in G-H} f ( h s )  = & \sum_{s \in S_{1}} f ( h s )   \noplus +
    \sum_{s \in S_{2}} f ( h s ) ,\\
    \sum_{s \in G-H} f ( x s^{-1} )  = & \sum_{s \in S_{1}} f ( x s^{-1} )
    \noplus + \sum_{s \in S_{2}} f ( x s^{-1} ) . \qedhere
  \end{align*}

\end{proof}

For a graph $\mathcal{G}= ( V,E )$, the complement graph is $\bar{\mathcal{G}}
= ( V, \bar{E} )$, where $( x,y ) \in \bar{E}$, if and only if $( x,y ) \nin
E$ for $x \neq y$. The eigenvector and eigenvalues of a $k$-regular graph and
its complement can be related, see for example Lemma 8.5.1 of
{\cite{Godsil2001}}. We present a relation between the of eigenvalues and
eigenvectors for pair-graphs $\mathcal{G} ( G,H,S_{1} )$ and $\mathcal{G} (
G,H,S_{2} )$ with $S_{1} \cup S_{2} =G-H$ and $S_{1} \cap S_{2} = \emptyset$.

First, assume that $| S_{1} | =k$, $| S_{2} | =n-k$, then any constant
function $f$ is an eigenfunction of both of the pair-graphs $\mathcal{G} (
G,H,S_{1} )$ and $\mathcal{G} ( G,H,S_{2} )$ corresponding to $\mu_{1} =k$ and
$\lambda_{1} =n-k$ respectively. Similarly, for $c \in \mathbbm{C}$, the
function $f=c ( \delta_{H} - \delta_{G-H} )$ is an eigenfunction of
$\mathcal{G} ( G,H,S_{1} )$ and $\mathcal{G} ( G,H,S_{2} )$ corresponding to
$\mu_{2n} =-k$ and $\lambda_{2n} =n-k$.

\begin{theorem}
  \label{symmetry}Let $S_{1}$ and $S_{2}$ as before with $| S_{1} | =k, |
  S_{2} | =n-k$ for numbers $0<k<n$, then
  \begin{itemizedot}
  \item (i) if $\mathcal{G} ( G,H,S_{1} )$ is not connected, then there are
    independent eigenfunctions $f$ and $g$ of $\mathcal{G} ( G,H,S_{1} )$
    associated to $\mu = \pm k$ such that $ f - g$ is an eigenfunction of of
    $\mathcal{G} ( G,H,S_{2} )$ corresponding to the eigenvalue $- \mu$.

  \item (ii) if $f$ is an eigenfunction of $\mathcal{G} ( G,H,S_{1} )$ \
    associated with an eigenvalue $\mu \neq \pm k$, then $f$ is an
    eigenfunction of $\mathcal{G} ( G,H,S_{2} )$ corresponding to the
    eigenvalue $- \mu$.
  \end{itemizedot}
\end{theorem}

\begin{proof}
  Denote the adjacency graph operator of graphs $\mathcal{G} ( G,H,S_{1} )
  $, $\mathcal{G} ( G,H,S_{2} )$, $\mathcal{G} ( G,H,G-H )$ as $A,B,C$
  respectively, so that $C = A+B$.

  {\tmem{(i)}} For a connected components $\Gamma$ of $\mathcal{G} (
  G,H,S_{1} )$, consider the function $f= \delta_{\Gamma}$ or $f=
  \delta_{\Gamma \cap H} - \delta_{\Gamma -H}$, which can be verified to be
  eigenfunctions corresponding to $\mu =k$ and $\mu =-k$, respectively. We can
  similarly define $g$ with respect to a different connected component
  $\Omega$ of $\mathcal{G} ( G,H,S_{1} )$. Note that there are no isolated
  vertices on the graph, therefore by Proposition \ref{comps}, all the
  connected components have the same cardinality, in particular $| \Gamma \cap
  H | = | \Omega \cap H |$ and $| \Gamma -H | = | \Gamma - \Omega |$. Then, we
  have
  \begin{align*}
    &C f ( h )  =  \sum_{x \in G-H} f ( x ) = \sum_{x \in \kappa -H} f
    ( x ) && \text{for }  h \in H,\\
    &C f ( x ) =  \sum_{h \in H} f ( h ) = \sum_{h \in \kappa \cap H} f ( h )
    && \text{for } x \in G-H.
  \end{align*}

  It follows that $C f =  \delta_{H} | \Gamma -H | + \delta_{G-H} | \Gamma
  \cap H |$ or $C f=- \delta_{H} | \Gamma -H | + \delta_{G-H} | \Gamma \cap H
  |$ depending on the whether $\mu =k$ or $\mu =-k$. By the preceding
  discussion, we have $C  ( f-g ) =0$ and
  \[ B  ( f-g ) =_{}   ( C-A ) ( f-g ) =C ( f-g ) -A ( f-g ) =- \mu ( f-g ) .
  \]
  {\tmem{(iii)}} For an eigenfunction $f$ of $A$ corresponding to an
  eigenvalue $\lambda \neq \pm k$, we have
  \[ B f ( x ) = ( C-A ) f ( x ) =C f ( x )  -  \mu f ( x ) , \]
  therefore it suffices to prove $C f ( x ) =0$. In other words, it suffices to
  prove the two equalities
  \[ \sum_{h \in H} f ( h ) =0, \]
  \[ \sum_{x \in G-H} f ( x ) =0. \]
  First, not that for fixed $s \in S_{1}$ we have
  \[ \sum_{h \in H} f ( h s ) = \sum_{x \in G-H} f ( x ) , \]
  then, by summing over all elements of $S_{1}$,
  \[ \sum_{s \in S} \sum_{h \in H} f ( h s ) = \sum_{s \in S} \sum_{x \in G-H}
  f ( x ) =k \left( \sum_{x \in G-H} f ( x ) \right) . \]
  Since $f$ is an eigenfunction associated with $\mu$, we have
  \[ \sum_{s \in S} \sum_{h \in H} f ( h s ) = \sum_{h \in H} \sum_{s \in S} f
  ( h s ) = \mu \left( \sum_{h \in H} f ( h ) \right) , \]
  therefore we have
  \begin{equation} \label{eqn:eqn1}
    k \left( \sum_{x \in G-H} f ( x ) \right) = \mu \left( \sum_{h \in H} f (
      h ) \right) .
  \end{equation}
  Similarly, by fixing $s \in S_{1}$ and using the equation
  \[ \sum_{x \in G-H} f ( x s^{-1} ) = \sum_{h \in H} f ( h ) , \]
  we obtain the relation
  \begin{equation}\label{eqn:eqn2}
    k \left( \sum_{h \in H} f ( h ) \right) = \mu \left( \sum_{x \in G-H} f (
      x ) \right) .
  \end{equation}
  Finally, combining (\ref{eqn:eqn1}) and (\ref{eqn:eqn2}), we obtain the equations
  \[ \sum_{h \in H} f ( h ) = \frac{\mu^{^{2}}}{k^{2}} \sum_{h \in H} f ( h )
  , \]
  \[ \sum_{x \in G-H} f ( x ) = \frac{\mu^{^{2}}}{k^{2}} \sum_{x \in G-H} f (
  x ) . \]
  Since $| \mu | \neq k$, we conclude \
  \[ \sum_{h  \in H} f ( h ) =0, \]
  \[ \sum_{x \in G-H} f ( x ) =0, \]
  and the proof is complete.
\end{proof}

Recall that when $\mathcal{G} ( G,H,S_{1} )$ is not connected, by Theorem
\ref{trivialeigen} \ the eigenvalue $\mu =k$ has multiplicity equal to the
number $c$ of connected components of $\mathcal{G} ( G,H,S_{1} )$, then by $(
i )$ the graph $\mathcal{G} ( G,H,S_{2} )$ also has the eigenvalue $m=k$ with
multiplicity $c-1$. We can reformulate the result as follows.

\begin{corollary}
  \label{symm}For group $G$ and subgroup $H$ of index $2$ and order $n$. Let
  $S \subset G-H $ with $| S | =k$ and $\lambda_{1} \geqslant \lambda_{2}
  \geqslant \ldots \geqslant \lambda_{2n}$ the spectrum of the pair-graph
  $\mathcal{G} ( G,H,S )$. Then there is a $( n-k )$-regular pair-graph
  $\mathcal{G} ( G,H,S' )$ with spectrum $\mu_{1} \geqslant \mu_{2} \geqslant
  \ldots \geqslant \mu_{2n}$ such that
  \begin{equation*}
    \lambda_{i} = \mu_{i}
  \end{equation*}
  for $i \neq 1,2n$.
\end{corollary}

\begin{example}
  For $G = \mathbbm{Z}/20\mathbbm{Z}$, $H=\mathbbm{Z}/10\mathbbm{Z}$, \ and
  $S_{1} = \{ 3,5,7 \}$, $S_{2} = \{ 1$, $3$, $5$, $13$, $15$, $17$, $19 \}$ we have the
  $\mathcal{G} ( G,H,S_{1} )$ and $\mathcal{G} ( G,H,S_{2} )$
  pair-graphs shown in figure \ref{fig:corrgraphs}. Table \ref{table1} contains the values of the positive eigenvalues
  for both of the graphs. Since both graphs are bipartite the remaining
  eigenvalues correspond to the negatives of the ones shown.

  \begin{figure}[h]
    \resizebox{3cm}{3cm}{\includegraphics{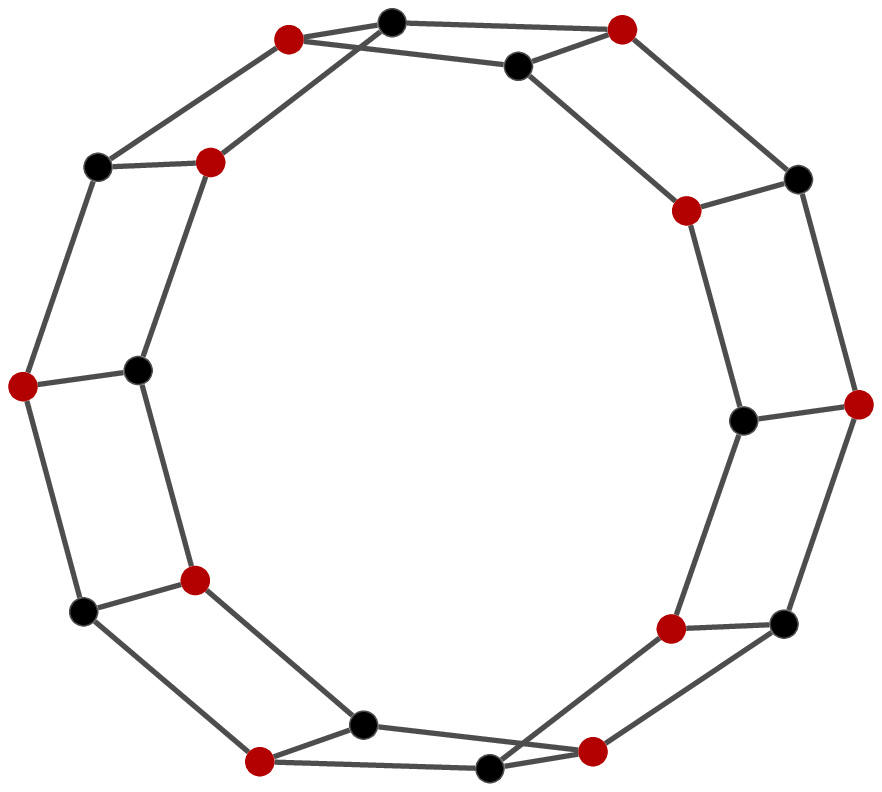}} \
    \resizebox{3cm}{3cm}{\includegraphics{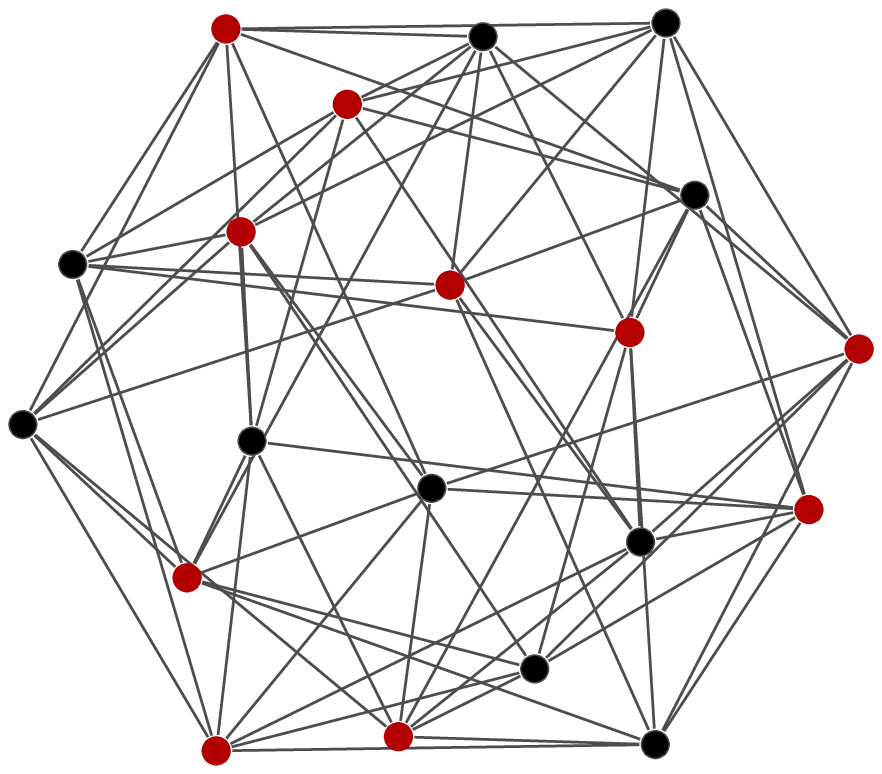}}
    \caption{Two $\mathcal{G} ( G,H,S )$ with the same nontrivial
      spectrum.}
    \label{fig:corrgraphs}
  \end{figure}

  \begin{table}[h]
    \begin{tabular}{|c|c|}
      \hline
      $\lambda_{i}$ & $\mu_{i}$\\
      \hline
      3 & 7\\
      $\frac{1}{2} \left( 3+ \sqrt{5} \right)$ & $\frac{1}{2} \left( 3+
        \sqrt{5} \right)$\\
      $\frac{1}{2} \left( 3+ \sqrt{5} \right)$ & $\frac{1}{2} \left( 3+
        \sqrt{5} \right)$\\
      $\frac{1}{2} \left( 1+ \sqrt{5} \right)$ & $\frac{1}{2} \left( 1+
        \sqrt{5} \right)$\\
      $\frac{1}{2} \left( 1+ \sqrt{5} \right)$ & $\frac{1}{2} \left( 1+
        \sqrt{5} \right)$\\
      1 & 1\\
      $\frac{1}{2} \left( -1+ \sqrt{5} \right)$ & $\frac{1}{2} \left( -1+
        \sqrt{5} \right)$\\
      $\frac{1}{2} \left( -1+ \sqrt{5} \right)$ & $\frac{1}{2} \left( -1+
        \sqrt{5} \right)$\\
      $\frac{1}{2} \left( 3- \sqrt{5} \right)$ & $\frac{1}{2} \left( 3-
        \sqrt{5} \right)$\\
      $\frac{1}{2} \left( 3- \sqrt{5} \right)$ & $\frac{1}{2} \left( 3-
        \sqrt{5} \right)$ \\
      \hline
    \end{tabular}
    \vspace{2mm}
    \caption{The positive values of the spectrum of the previous
      graphs.}
    \label{table1}
  \end{table}
  Note that $S_{1} \cup S_{2} \neq G-H= \{ 1,3,5, \ldots ,19 \}$, but $S'
  =R_{4} ( S_{1} ) = \{ 7,9,11 \}$ is such that $\mathcal{G} ( G,H,S_{1} )
  \cong \mathcal{G} ( G,H,S' )$ by Proposition \ref{rigact}, and $S' \cup
  S_{2} =G-H$.
\end{example}

{\remark{Corollary \ref{symm} defines a symmetric relation between the
    nontrivial spectrum of graphs for complementary choices of $S$ and the
    previous example shows that for a $k$-regular graph $\mathcal{G} ( G,H,S )$,
    there can be more than one $( n-k )$-regular $\mathcal{G} ( G,H,S' )$ graphs
    with the same nontrivial spectrum. In fact, if we find one, by using
    Proposition \ref{rigact} with right actions of $H$ we can obtain families of
    graphs with the same nontrivial spectrum. In order for the relation to be one
    to one, we need to consider equivalence classes of $\mathcal{G} ( G,H,S )$
    graphs under graph isomorphism. An unsettled question is whether the number of
    equivalence classes of $k$-regular $\mathcal{G} ( G,H,S )$ is the same as that
    of $( n-k )$-regular ones for fixed group $G$ and subgroup $H$ of index 2.}}

Recall that a Ramanujan graph is a connected $k$-regular graph such that for
every eigenvalue $\mu$ different from $\pm k$, one has
\[ | \mu | \leqslant 2 \sqrt{k-1} . \]
\begin{corollary}
  \label{eigenbound}With the same assumptions as before. If the pair
  graph $\mathcal{G} ( G,H,S )$ is connected and
  \[ | S | \geqslant n+2 -2 \sqrt{n},  \]
  then it is a bipartite Ramanujan graph.
\end{corollary}

\begin{proof}
  Due to Theorem \ref{symmetry} any $k$-regular pair-graph $\mathcal{G} (
  G,H,S )$ has nontrivial eigenvalues $\mu$ satisfying $| \mu | \leqslant \min
  \{ k,n-k \}$. Also, the pair-graph $\mathcal{G} ( G,H,S )$ is a Ramanujan
  graph when said trivial eigenvalues satisfy $| \mu | \leqslant 2
  \sqrt{k-1}$. Considering the two inequalities, it follows that all
  $k$-regular pair-graphs $\mathcal{G} ( G,H,S )$ with $k \leqslant 2$ or $k
  \geqslant n+2-2 \sqrt{n}$ are Ramanujan graphs.
\end{proof}

{\remark{\label{remark5}Note that the condition is not necessary. Also, since
    the bound $n+2-2 \sqrt{n}$ grows with the size of $n$ it is not immediately
    useful for finding families of Ramanujan graphs, but it may be useful for the construction and verification of Ramanujan graphs.}}

\begin{example}
  Let $G=\mathfrak{S}_{4}$, $H=A_{4}$. The set $S= \{ ( 1,2 )$, $( 1,3 )$, $(
  2,4 )$, $( 3,4 )$, $( 1,2,3,4 )$, $( 1,3,2,4 )$, $( 1,4,2,3 )$, $( 1,4,3,2 ) \}$
  is such that $| S | =8$ satisfy the bound of the corollary, so the
  corresponding $\mathcal{G} ( G,H,S )$ graph is Ramanujan. Its spectrum consist of $\pm 8$ with
  multiplicity 1, $\pm 4$ with multiplicity 2 and 0 with multiplicity
  18. Note that it contains the eigenvalues $\pm 4$ and as expected by
  Theorem \ref{symmetry}, the corresponding $4$-regular graph,
  generated by $S' = \{ ( 1,2 )$, $(3,4 )$, $ (1,3,2,4 )$, $( 1,4,2,3
  )\}$ has 3 connected components. Both of the pair-graphs are shown in
  figure \ref{fig:ramgcor}.

  \begin{figure}
    \resizebox{4cm}{4cm}{\includegraphics{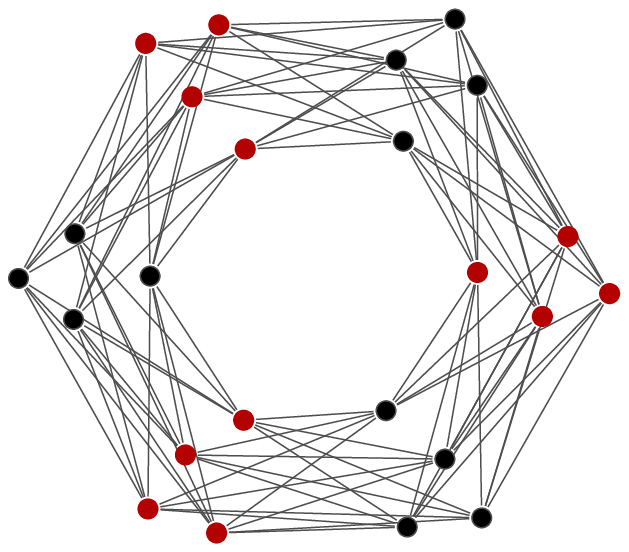}} \
    \resizebox{4cm}{4cm}{\includegraphics{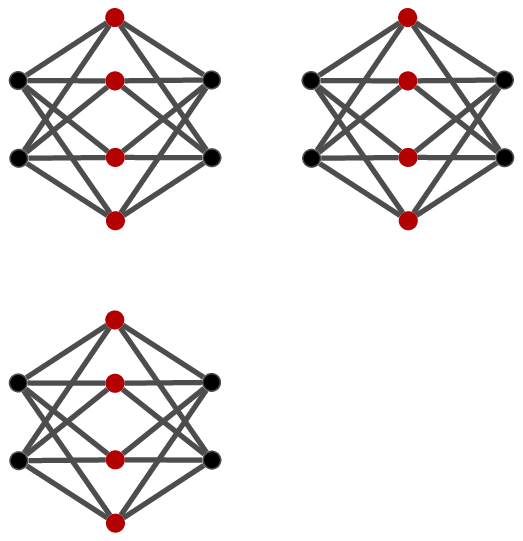}}
    \caption{A Ramanujan pair-graph and its corresponding pair-graph.}
    \label{fig:ramgcor}
  \end{figure}
\end{example}

\begin{example}
  Let $G= \tmop{GL}_{2} ( \mathbbm{F}_{3} )$ and $H= \tmop{SL}_{2} (
  \mathbbm{F}_{3} )$. Then $| G | =48$ and $| H | =24$, in this case by
  Corollary \ref{eigenbound}, for any subset $S$ with $| S | \geqslant 17$,
  the resulting pair-graph $\mathcal{G} ( G,H,S )$ graph is Ramanujan. For a
  random choice of $S$ we obtain the $17$-regular Ramanujan pair
  graph $\mathcal{G} ( G,H,S )$ shown in figure \ref{fig:ramgff}.

  \begin{figure}[h]
    \resizebox{5cm}{5cm}{\includegraphics{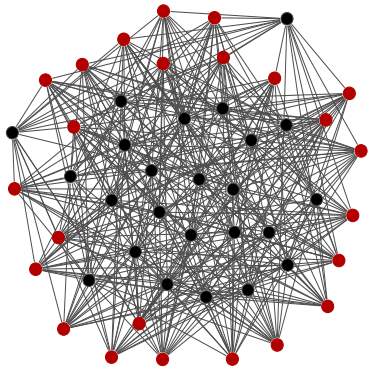}}
    \caption{ The Ramanujan pair-graph $\mathcal{G} ( G,H,S )$.}
    \label{fig:ramgff}
  \end{figure}

  In this case, the pair-graph $\mathcal{G} ( G,H, ( G-H ) -S )$ of
  figure \ref{ramgraph2}, generated by the complement of $S$ in $G-H$,
  is a $7$-regular Ramanujan graph. This shows that the result of
  Corollary \ref{eigenbound} is not a necessary condition, as
  mentioned in the above remark.

  \begin{figure}[h]
    \resizebox{4cm}{4cm}{\includegraphics{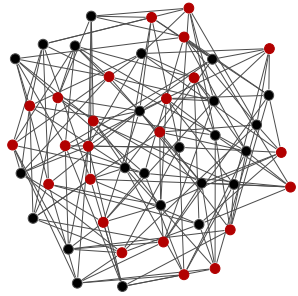}}
    \caption{The Ramanujan pair-graph $\mathcal{G} ( G,H, ( G-H ) -S
      )$.}
    \label{ramgraph2}
  \end{figure}
\end{example}

As mentioned in the Introduction, Ramanujan graphs are graphs with good
connectivity (as measured by the isoperimetric constant of the graph). Since
the addition of edges to a graph results in a graph with greater or equal
connectivity, it is expected for graphs with large amount of vertices to have
good connectivity. This can be confirmed in the above examples, as the
Ramanujan pair-graphs $\mathcal{G} ( G,H,S )$ resulting from using the bound
provided by Corollary $\ref{eigenbound}$, for group $G$ and subgroup $H$ of
order $n \geqslant 4$, have generating set $S$ satisfying $| S | \geqslant
n/2$ . Then, Corollary $\ref{eigenbound}$ can also be stated as follows: The
set of graphs
\[ \{ \mathcal{G} ( G,H,S )  | S \subset G-H, | S | \geqslant m \} \]
consists only of Ramanujan graphs when $m=n+2-2 \sqrt{n}$, where $H$ is a
subgroup of order $n$ and $[ G:H ] =2$.

\section{Summary}

\begin{figure}[h]
  \label{diagramsummary}
  \includegraphics{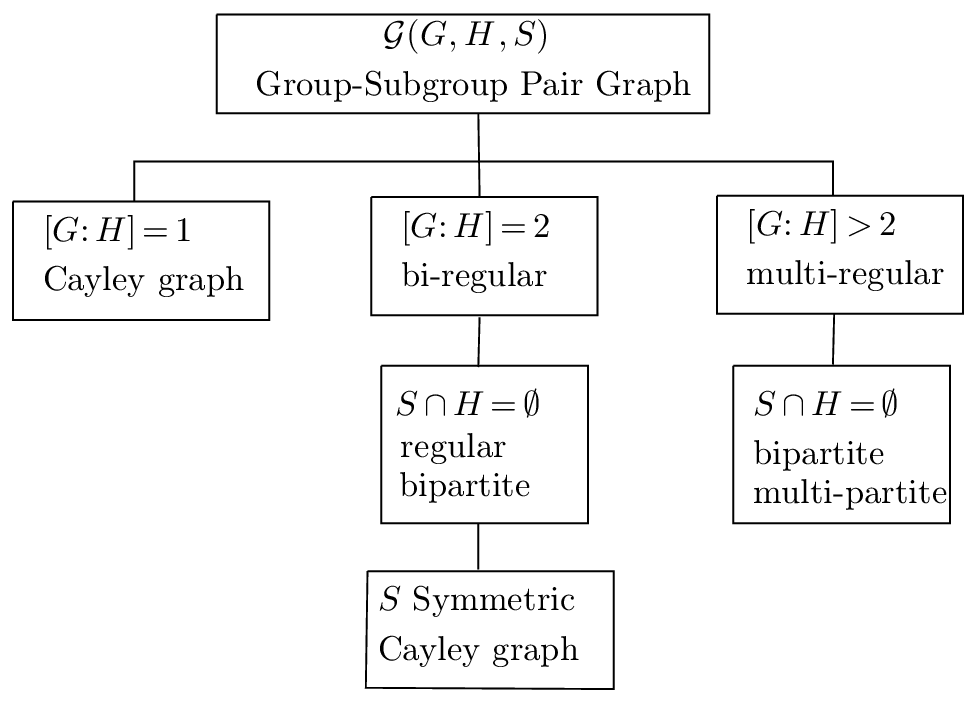}
  \caption{Characteristics of group-subgroup pair graphs.}
\end{figure}

The general properties of group-subgroup pair graphs
according to the choice of $G,H$ and $S= \emptyset$ are shown in diagram {\ref{diagramsummary}}.
 For arbitrary pair-graph $\mathcal{G} ( G,H,S )$ a set of trivial eigenvalues can be determined
 including \ the largest eigenvalue of the graph. According to the index
 of the subgroup the pair-graph can be regular (index 1 or 2) or
never be regular. If the generating set is contained in complement of the
subgroup, the pair-graph is bipartite or multipartite. The pair-graphs are
Cayley graphs in two cases: when the subgroup and the group are equal and when
the subgroup is of order 2 and the graphs is regular bipartite with symmetric
generating set. The symmetry on the generating set also makes easier to verify
if the given pair-graph is bipartite. These observations suggest that further
research on the matter may be started with graphs with symmetric subsets
before considering the general case.

\section{Acknowledgments}

The author was supported by the Japanese Government (MONBUKAGAKUSHO: MEXT)
Scholarship.

The author would like to express his gratitude to Professor Masato Wakayama
for providing the motivation for the present work and for the invaluable
discussions and support. The author would also like to thank Professors
Yoshinori Yamasaki, \ Miki Hirano and Kazufumi Kimoto \ for their valuable
comments and discussion, regarding the definition of the Group-Subgroup pair
graph during the event ``Zeta Functions in Okinawa 2013''; and Mr. Hiroto
Inoue for his valuable suggestions on the statements and proofs of some of the
results.

The computations and diagrams were elaborated using Mathematica 9.0 Student
Edition. The source files for the diagrams can be downloaded from\\
\center{\href{http://www2.math.kyushu-u.ac.jp/~ma213054/files/figures.nb}{http://www2.math.kyushu-u.ac.jp/\~{}ma213054/files/figures.nb}.}

\begin{flushleft}
  Cid Reyes-Bustos \par
  Graduate School of Mathematics \par
  Kyushu University \par
  744 Motooka, Nishi-ku, Fukuoka, 819-0395 JAPAN \par
  \tmverbatim{ma213054@math.kyushu-u.ac.jp} \par
\end{flushleft}

\end{document}